\documentclass{birkjour}
 \newtheorem{thm}{Theorem}[section]
 \newtheorem{cor}[thm]{Corollary}
 \newtheorem{lem}[thm]{Lemma}
 \newtheorem{prop}[thm]{Proposition}
 \theoremstyle{definition}
 \newtheorem{defn}[thm]{Definition}
 \newtheorem{ex}[thm]{Example}
 \theoremstyle{remark}
 \newtheorem{rem}[thm]{Remark}
 \numberwithin{equation}{section}
\usepackage{color}
\begin{document}
%
%
%
%
%
%
%
%
%

\title[Weaving Hilbert space fusion  frames]
 {Weaving Hilbert Space Fusion Frames}


\author[F. Arabyani Neyshaburi]{Fahimeh Arabyani Neyshaburi}
\address{Department of Mathematics and Computer Sciences, Hakim Sabzevari University, Sabzevar, Iran.}
\email{fahimeh.arabyani@gmail.com}
\author[A. Arefijamaal]{Ali Akbar Arefijamaal}
\address{Department of Mathematics and Computer Sciences, Hakim Sabzevari University, Sabzevar, Iran.}
\email{arefijamaal@hsu.ac.ir}


\begin{abstract}
A new notion in frame theory, so called  weaving  frames has  been recently introduced   to deal with some problems in signal processing and wireless sensor networks. Also, fusion frames are an important extension of frames, used in many areas especially for wireless sensor networks.
In this paper,  we survey the notion of weaving Hilbert space fusion frames. This concept can be had  potential applications in wireless  sensor networks which require distributed processing using different fusion frames. Indeed, we present several approaches for identifying and constructing of weaving fusion frames in terms of local frames,  bounded operators in   Hilbert spaces and also dual fusion frames.
To this  end, we present some conditions under  which a fusion frame with its duals constitute some pair of woven fusion frames. As a result, we show that Riesz fusion bases are woven with all of their duals. Finally,  we obtain some new results  on  fusion frames and weaving fusion frames under operator perturbations.
\end{abstract}

\maketitle

\section{Introduction and preliminaries}

\smallskip

Fusion frame theory is a fundamental mathematical theory   introduced in  \cite{Cas04} to model sensor networks perfectly.
Although, recent studies shows that fusion frames provide effective frameworks  not only for modeling of sensor networks but also   for signal and image processing, sampling theory, filter banks and a variety of applications that cannot be modeled by discrete frames \cite{Cas08, sensor, hear}. In the following, we review basic definitions and  results of fusion frames.

 Let $\{W_i\}_{i\in I}$ be a family of closed subspaces of $\mathcal{H}$ and $\{\omega_i\}_{i\in I}$ a family of
weights, i.e. $\omega_i>0$, $i\in I$. Then $W:=\{(W_i,\omega_i)\}_{i\in
I}$ is called a \textit{fusion frame} for $\mathcal{H}$ if there exist the
constants $0<C_{W}\leq D_{W}<\infty$ such that
\begin{eqnarray}\label{Def. fusion}
C_{W}\|f\|^{2}\leq \sum_{i\in I}\omega_i^2\|\pi_{W_i}f\|^2\leq
D_{W}\|f\|^{2},\qquad (f\in \mathcal{H}),
\end{eqnarray}
where $\pi_{W_{i}}$ denotes the orthogonal projection from Hilbert space $\mathcal{H}$ onto a closed subspace $W_{i}$. The constants $C_{W}$ and $D_{W}$ are called the \textit{fusion frame
bounds}. If $W$ is a fusion frame for $\overline{\operatorname{span}}\{W_{i}\}_{i\in I}$ then $W$ is called a \textit{fusion frame sequence} and
if $\omega_i=1$, for all $i$, $W$ is called $1$-uniform fusion frame and we denote it by $\lbrace
W_{i}\rbrace_{i\in I}$. Also, if we only have the upper bound in (\ref{Def. fusion}) we
call $\{(W_i,\omega_i)\}_{i\in I}$ a \textit{Bessel fusion
sequence}.
A family of closed subspaces $\{W_i\}_{i\in I}$ is called an orthonormal basis for $\mathcal{H}$ when $\oplus_{i\in I} W_{i} = \mathcal{H}$, and it is  called a \textit{fusion Riesz sequence} whenever  there exist positive constants $C_{W}$, $D_{W}$ such that for every finite subset $J\subset I$ and arbitrary vector $f_{i}\in W_i$, we have
\begin{eqnarray*}\label{Def fusion}
C_{W}\sum_{i\in J}\| f_i\|^{2} \leq \| \sum_{i\in J} f_i\|^{2} \leq D_{W}\sum_{i\in J}\| f_i\|^{2},
\end{eqnarray*}
moreover, if $W$ is a complete family  in  $\mathcal{H}$ it is called a \textit{fusion Riesz basis}.
It is clear that every  orthonormal basis of subspaces and also every  fusion Riesz basis is a    $1$-uniform fusion frame for $\mathcal{H}$.
Recall that for each sequence $\{W_i\}_{i\in I}$ of closed subspaces
in $\mathcal{H}$, the space
\begin{eqnarray*}
\sum_{i\in I}\oplus W_{i} =\left\{\{f_i\}_{i\in I}:f_i\in W_i,
\sum_{i\in I}\|f_i\|^2<\infty\right\},
\end{eqnarray*}
 with the inner product $\left\langle \{f_i\}_{i\in I},\{g_i\}_{i\in I} \right\rangle=\sum_{i\in
I}\langle f_i,g_i \rangle$ is a Hilbert space.

For a Bessel fusion sequence $W := \{(W_i,\omega_i)\}_{i\in I}$ of
$\mathcal{H}$, the \textit{synthesis operator} $T_{W}: \sum_{i\in
I}\oplus W_{i} \rightarrow\mathcal{H}$ is defined by
\begin{equation*}
T_{W}\left(\{f_i\}_{i\in I}\right)=\sum_{i\in I}\omega_if_i,\qquad
\left(\{f_{i}\}_{i\in I}\in \sum_{i\in I}\oplus W_{i}\right).
\end{equation*}
Its adjoint operator $T_{W}^{*}: \mathcal{H}\rightarrow \sum_{i\in
I}\oplus W_{i}$, which is called the \textit{analysis
operator}, is given by
\begin{eqnarray*}
T_{W}^{*}(f)=\left\{\omega_{i}\pi_{W_{i}}(f)\right\}_{i\in I},\qquad (f\in
\mathcal{H}).\end{eqnarray*}
and the \textit{fusion frame operator}
$S_{W}:\mathcal{H}\rightarrow\mathcal{H}$ is defined by $S_{W
}f=\sum_{i\in I}\omega_i^{2}\pi_{W_i}f$, which is a bounded,
invertible and positive operator \cite{Cas04}.
For each $i\in I$, let $W_{i}$ be a closed subspace of $\mathcal{H}$ and $\omega_{i}>0$. Also,  let $\lbrace f_{i,j}\rbrace_{j\in J_{i}}$ be a frame for $W_{i}$ with frame bounds $\alpha_{i}$ and $\beta_{i}$ such that
\begin{eqnarray}\label{sup}
 0 < \alpha=inf_{i\in I} \alpha_{i} \leq \beta=sup_{i\in I}\beta_{i} < \infty.
\end{eqnarray}
For every fusion frame as $\lbrace (W_{i},\omega_{i})\rbrace_{i\in I}$, there exist frames $\lbrace f_{i,j}\rbrace_{j\in J_{i}}$ for $W_{i}$, so that satisfies (\ref{sup}). These frames are called the \textit{local frames} of $\lbrace (W_{i},\omega_{i})\rbrace_{i\in I}$ and it is well known that $\lbrace (W_{i},\omega_{i})\rbrace_{i\in I}$ is a fusion frame if and only if $\lbrace \omega_{i}f_{i,j}\rbrace_{j\in J_{i}, i\in I}$ is a frame for $\mathcal{H}$,  see Theorem 3.2 of \cite{Cas04}.

First definition of dual fusion frames  was presented by P.
G$\breve{\textrm{a}}$vru\c{t}a in \cite{Gav07}.
A Bessel fusion sequence $\{(V_i,\nu_i)\}_{i\in I}$ is
called a \textit{dual} fusion frame of $W=\{(W_i,\omega_i)\}_{i\in I}$ if
\begin{eqnarray}\label{Def:alt}
f=\sum_{i\in I}\omega_{i}\nu_{i}\pi_{V_i}S_{W}^{-1}\pi_{W_i}f,\qquad
(f\in \mathcal{H}).
\end{eqnarray}
The family $\{(S_{W}^{-1}W_i,\omega_i)\}_{i\in I}$, which is a dual
fusion frame of $W$, is called the \textit{canonical dual} of
$\{(W_i,\omega_i)\}_{i\in I}$.  Since the explicit computations of dual fusion frames is intricate, the authors in \cite{arefi1} introduced and investigated  the notion of approximate duals for fusion frames to obtain some characterizations of dual fusion frames. Indeed, a Bessel fusion sequence $\lbrace (V_{i}, \upsilon_{i})\rbrace_{i\in I}$ is called an approximate  dual of a  fusion frame $W=\lbrace (W_{i}, \omega_{i})\rbrace_{i\in I}$ if
$\Vert f- \sum_{i\in I}\omega_{i}\nu_{i}\pi_{V_i}S_{W}^{-1}\pi_{W_i}f\Vert < 1$, for every $f\in \mathcal{H}$, see \cite{arefi1, Gav07, Hei15, Hei14} for more details on dual and approximate dual fusion frames.

\smallskip

Recently,
  Bemrose et al.  \cite{Bemros} have introduced the notion of \textit{weaving Hilbert space frames} due to some new  problems  arising in distributed  signal processing and wireless sensor networks.
In the sequel,  we state some preliminaries of weaving frames that are used in our main results.

\begin{defn}
A finite family of frames $\{\phi_{ij}\}_{j=1,i\in I}^{M}$ in Hilbert space $\mathcal{H}$ is said  to be woven if there are universal constants $A$ and $B$ so that for every partition $\{\sigma_{j}\}_{j=1}^{M}$ of $I$, the family   $\{\phi_{ij}\}_{j=1,i\in \sigma_{j}}^{M}$ is a frame for $\mathcal{H}$ with bounds $A$ and $B$, respectively. Each family $\{\phi_{ij}\}_{j=1,i\in \sigma_{j}}^{M}$ is called a weaving.
\end{defn}
Moreover the family $\{\phi_{ij}\}_{j=1,i\in I}^{M}$ is called weakly woven if for every partition  $\{\sigma_{j}\}_{j=1}^{M}$ of $I$, the family $\{\phi_{ij}\}_{j=1,i\in \sigma_{j}}^{M}$ is a frame for $\mathcal{H}$.
\begin{thm}\label{2.0}\cite{Bemros}
given two frames  $\{\varphi_{i}\}_{i\in I}$ and $\{\psi_{i}\}_{i\in I}$ for $\mathcal{H}$, the following are equivalent:
\begin{itemize}
\item[(i)] The two frames are woven.
\item[(ii)] The two frames are weakly woven.
\end{itemize}
 \end{thm}

\begin{prop}\label{prop1.4}
Let $\varphi=\{\varphi_{i}\}_{i\in I}$ be a frame and $U$ be an invertible operator satisfying
$\Vert I_{\mathcal{H}} - U\Vert  < \dfrac{A_{\varphi}}{B_{\varphi}}$.
Then $\varphi$ and $U\varphi$ are woven.
 \end{prop}
For more details on weaving frames and some generalization of weaving frame theory see \cite{Bemros, lynch, Deepshika0, Deepshika, Vashisht}. In this note, we present several approaches for identifying and constructing of weaving fusion frames in terms of local frames,  bounded operators in   Hilbert spaces and specially dual fusion frames. Moreover, we give  some new results  on  fusion frames and weaving fusion frames under operator perturbations.

Throughout the paper, we suppose  $\mathcal{H}$ is a separable Hilbert space, $I$ a countable index set and $I_{\mathcal{H}}$ is the identity operator on $\mathcal{H}$. For  every $\sigma\subset I$, we show the complement of $\sigma$ by $\sigma^{c}$. Also,  we use of $[n]$ to denote the set $\{1, 2, ..., n\}$.  For two Hilbert spaces $\mathcal{H}_{1}$ and $\mathcal{H}_{2}$ we denote by $B(\mathcal{H}_{1},\mathcal{H}_{2})$ the collection of all bounded linear operators between $\mathcal{H}_{1}$ and $\mathcal{H}_{2}$, and we abbreviate $B(\mathcal{H},\mathcal{H})$ by $B(\mathcal{H})$. Moreover,  we denote the range of $T\in B(\mathcal{H})$ by $R(T)$, the null space of $T$ by $N(T)$ and the orthogonal projection of $\mathcal{H}$ onto a closed subspace $V \subseteq \mathcal{H}$  by $\pi_{V}$.

\section{Weaving fusion frames}
In this section, we survey the notion of weaving Hilbert space fusion frames. Also, we obtain some approaches for constructing of weaving fusion frames and Riesz fusion bases in terms of local frames and   dual fusion frames. As a result we show that Riesz fusion bases are woven with all of their duals.
\begin{defn}
A finite family of fusion frames $\{(W_{ij},\omega_{ij})\}_{j=1,i\in I}^{M}$ in Hilbert space $\mathcal{H}$ is said  to be woven if there are universal constants $C$ and $D$ so that for every partition $\{\sigma_{j}\}_{j=1}^{M}$ of $I$, the family   $\{(W_{ij},\omega_{ij})\}_{j=1,i\in \sigma_{j}}^{M}$ is a fusion frame for $\mathcal{H}$ with bounds $C$ and $D$, respectively. Each family $\{(W_{ij},\omega_{ij})\}_{j=1,i\in \sigma_{j}}^{M}$ is called a weaving.
\end{defn}
Also, the family $\{(W_{ij},\omega_{ij})\}_{j=1,i\in I}^{M}$ is called weakly woven if for every partition  $\{\sigma_{j}\}_{j=1}^{M}$ of $I$, the family $\{(W_{ij},\omega_{ij})\}_{j=1,i\in \sigma_{j}}^{M}$ is a fusion frame for $\mathcal{H}$. Clearly, for two fusion frame $W$ and $V$ every weaving $\{(W_{ij},\omega_{ij})\}_{j=1,i\in \sigma_{j}}^{M}$  is a Bessel fusion sequence with Bessel bound $\sum_{j\in [n]}D_{W_{j}}$.

The following lemma, which gives a necessary and sufficient condition for weaving fusion frames in terms of local frames.
\begin{lem}\label{1lem}
For every $j\in [n]$, suppose $W_{j}=\{(W_{ij},\omega_{ij})\}_{i\in I}$  is a  fusion frame  for $\mathcal{H}$. Also, let $\{f_{i_{j},k}\}_{k\in K_{i_{j}}}$ be respective local frames of $W_{ij}$. Then the following conditions are equivalent;
\begin{itemize}
\item[(i)] The family $\{W_{j} :\quad j\in [n]\}$ are woven.
\item[(ii)] The family $\cup_{j\in [n]}\{\omega_{ij}f_{i_{j},k}\}_{k\in K_{i_{j}}, i\in \sigma_{j}}$ is a frame for $\mathcal{H}$.
\end{itemize}
\end{lem}
\begin{proof}
Consider $A_{i_{j}}$ and  $B_{i_{j}}$ as the lower and upper frame bounds of $\{f_{i_{j},k}\}_{k\in K_{i_{j}}}$, respectively. Then for any partition $\sigma=\{\sigma_{j}\}_{j\in [n]}$ of $I$  we have that
\begin{eqnarray*}
\mathcal{A}:=\inf \{A_{i_{j}}; \quad  j\in [n], i\in \sigma_{j}\}\geq \min \{\inf\{A_{i_{j}}\}_{i\in I}, \quad j\in [n]\}>0,
\end{eqnarray*}
and
\begin{eqnarray*}
\mathcal{B}:=\sup\{B_{i_{j}}; \quad  j\in [n], i\in \sigma_{j}\}\leq \max \{\sup\{B_{i_{j}}\}_{i\in I}, \quad j\in [n]\}< \infty.
\end{eqnarray*}
Hence, for every weaving $\cup_{j\in [n]}\{(W_{ij}, \omega_{ij})\}_{i\in \sigma_{j}}$ the sequences
\begin{eqnarray*}
\{h_{i_{j}, k}\}_{k} =\begin{cases}
\begin{array}{ccc}
\{f_{i_{1},k}\}_{k\in K_{i_{1}}}& \;
{i\in \sigma_{1}}, \\
\{f_{i_{2},k}\}_{k\in K_{i_{2}}}& \;
{i\in \sigma_{2}}, \\
.\\
.\\
.\\
\{f_{i_{n},k}\}_{k\in K_{i_{n}}}& \;
{i\in \sigma_{n}}, \\
\end{array}
\end{cases}
\end{eqnarray*}
are local frames.
Thus, applying Theorem 3.2 of \cite{Cas04}, $\cup_{j\in [n]}\{(W_{ij}, \omega_{ij})\}_{i\in \sigma_{j}}$ is a fusion frame if and only if  $\cup_{j\in [n]}\{\omega_{ij}f_{i_{j},k}\}_{k\in K_{i_{j}}, i\in \sigma_{j}}$ is a frame for $\mathcal{H}$. This follows the desired result.
\end{proof}
The next theorem is proved by a similar approach to Theorem \ref{2.0} and  is useful in some of our main result..
\begin{thm}\label{weakly equ}
Suppose $W=\{(W_{i},\omega_{i})\}_{i\in I}$ and $V=\{(V_{i},\nu_{i})\}_{i\in I}$ are two fusion frames  for $\mathcal{H}$, the following are equivalent;
\begin{itemize}
\item[(i)] The two fusion frames $W$ and $V$  are woven.
\item[(ii)] The two fusion frames $W$ and $V$  are weakly woven.
\end{itemize}
\end{thm}

Applying  Lemma \ref{1lem}  we  obtain the following   extension  of    Theorem 5.2  in \cite{Bemros} for  Riesz fusion bases.
\begin{thm}\label{riesz1}
Let $\lbrace
W_{i}\rbrace_{i\in I}$ and  $\lbrace
V_{i}\rbrace_{i\in I}$ be fusion Riesz bases for $\mathcal{H}$  so that for  every $\sigma \subset I$,   the family $\{W_{i}\}_{i\in \sigma}\cup\{V_{i}\}_{i\in \sigma^{c}}$ is a   fusion Riesz sequence. Then for every $\sigma \subset I$  the family $\{W_{i}\}_{i\in \sigma}\cup\{V_{i}\}_{i\in \sigma^{c}}$ is a  fusion Riesz basis for $\mathcal{H}$.
\end{thm}
\begin{proof}
Suppose $\{e_{i,j}\}_{j\in J_{i}}$ and $\{u_{i,j}\}_{j\in K_{i}}$ are respective orthonormal bases of $W_{i}$ and $V_{i}$, for all $i\in I$. Then $\{e_{i,j}\}_{j\in J_{i}, i\in I}$ and $\{u_{i,j}\}_{j\in K_{i}, i\in I}$ are Riesz bases for $\mathcal{H}$, see \cite{Cas04}. Moreover,  the family  $\{e_{i,j}\}_{j\in J_{i}, i\in \sigma}\cup \{u_{i,j}\}_{j\in K_{i}, i\in \sigma^{c}}$ is Riesz sequence, for every $\sigma \subset I$.
More precisely, let $C_{\sigma}$ and $D_{\sigma}$ be  Riesz sequence bounds of $\{W_{i}\}_{i\in \sigma}\cup\{V_{i}\}_{i\in \sigma^{c}}$, respectively  and  $\{c_{i,j}\}_{j\in J_{i}, i\in \sigma}\cup \{d_{i,j}\}_{j\in K_{i}, i\in \sigma^{c}}$ be a sequence in $l^{2}$. Then  we obtain
 \begin{eqnarray*}
&&C_{\sigma}\left(\sum_{ i\in \sigma}\sum_{j\in J_{i}}\left\vert c_{i,j}\right\vert^{2}+ \sum_{ i\in \sigma^{c}}  \sum_{j\in K_{i}}\left\vert d_{i,j}\right\vert^{2}\right)\\
&=&C_{\sigma}\left(\sum_{ i\in \sigma} \left\Vert \sum_{j\in J_{i}}c_{i,j}e_{i,j}\right\Vert^{2}+ \sum_{ i\in \sigma^{c}} \left\Vert \sum_{j\in K_{i}}d_{i,j}u_{i,j}\right\Vert^{2}\right)\\
&\leq& \left \Vert \sum_{ i\in \sigma}\sum_{j\in J_{i}}c_{i,j}e_{i,j}+\sum_{ i\in \sigma^{c}}\sum_{j\in K_{i}}d_{i,j}u_{i,j}\right\Vert^{2}\\
&\leq& D_{\sigma}\left(\sum_{ i\in \sigma} \left\Vert \sum_{j\in J_{i}}c_{i,j}e_{i,j}\right\Vert^{2}+ \sum_{ i\in \sigma^{c}} \left\Vert \sum_{j\in K_{i}}d_{i,j}u_{i,j}\right\Vert^{2}\right)\\
&=&D_{\sigma}\left(\sum_{ i\in \sigma}\sum_{j\in J_{i}}\left\vert c_{i,j}\right\vert^{2}+ \sum_{ i\in \sigma^{c}}  \sum_{j\in K_{i}}\left\vert d_{i,j}\right\vert^{2}\right)\\
\end{eqnarray*}
 Hence,  by a similar approach to Lemma 5.1 and Theorem 5.2 of \cite{Bemros} one can check that  the family $\{e_{i,j}\}_{j\in J_{i}, i\in \sigma}\cup \{u_{i,j}\}_{j\in K_{i}, i\in \sigma^{c}}$ is actually a Riesz basis and consequently $\{W_{i}\}_{i\in \sigma}\cup\{V_{i}\}_{i\in \sigma^{c}}$ is a  fusion Riesz basis for $\mathcal{H}$, by Lemma \ref{1lem}.
\end{proof}
The next result is also proved   by Lemma \ref{1lem} and with a similar approach to Theorem 5.3 of \cite{Bemros}.
\begin{prop}
Let $\lbrace
W_{i}\rbrace_{i\in I}$ and  $\lbrace
V_{i}\rbrace_{i\in I}$ be fusion Riesz bases and there is a uniform constant $A>0$  so that for every $\sigma \subset I$ the family $\{W_{i}\}_{i\in \sigma}\cup\{V_{i}\}_{i\in \sigma^{c}}$ is a  fusion  frame with lower bound $A$. Then for every $\sigma \subset I$  the family $\{W_{i}\}_{i\in \sigma}\cup\{V_{i}\}_{i\in \sigma^{c}}$ is a  fusion Riesz basis for $\mathcal{H}$.
\end{prop}
\begin{rem}
It is worth to note that unlike discrete Riesz bases, see Theorem 5.4 of  \cite{Bemros}, if
$W=\lbrace
W_{i}\rbrace_{i\in I}$ is  a fusion Riesz basis and  $V=\lbrace
(V_{i}, \nu_{i})\rbrace_{i\in I}$ is a fusion frame for $\mathcal{H}$ so that $W$ and $V$ are woven. Then $V$ is not necessary  a  fusion Riesz basis for $\mathcal{H}$. Indeed, Suppose
\begin{eqnarray*}
W_{1} ={\textit{span}}\{(1,0,0)\}, \quad W_{2} = {\textit{span}}\{(0,1,0)\}, \quad W_{3} = {\textit{span}}\{(0,0,1)\}.
\end{eqnarray*}
and
\begin{eqnarray*}
V_{1} ={\textit{span}}\{(1,0,0), (0,1,0)\}, \quad V_{2} = {\textit{span}}\{(0,1,0)\}, \quad V_{3} = {\textit{span}}\{(0,0,1)\}.
\end{eqnarray*}
Then $W = \lbrace W_{i}\rbrace_{i=1}^{3}$ is a fusion Riesz basis of $\mathbb{R}^{3}$ which is woven with $V= \lbrace V_{i}\rbrace_{i=1}^{3}$, while  $V$ is not fusion Riesz basis.
\end{rem}
\begin{thm}
Let $W=\lbrace
W_{i}\rbrace_{i\in I}$   be a fusion Riesz basis for $\mathcal{H}$. Then there exists a fusion Riesz basis $\lbrace
V_{i}\rbrace_{i\in I}$ so that $W_{i}\perp V_{j}$, for all $i\neq j$. Moreover, for every $\sigma \subset I$  the family $\{W_{i}\}_{i\in \sigma}\cup\{V_{i}\}_{i\in \sigma^{c}}$ is a fusion Riesz basis for $\mathcal{H}$,
i.e., $W$ and $V$ are woven.
\end{thm}
\begin{proof}
Since $W$ is a fusion Riesz  basis so there is an  invertible operator $U\in B(\mathcal{H})$ and an orthonormal fusion basis $\lbrace
N_{i}\rbrace_{i\in I}$ so that $W_{i}=UN_{i}$, for all $i\in I$, see \cite{Asgari}. Consider $V_{i}=(U^{-1})^{*}N_{i}$  then $V=\lbrace
V_{i}\rbrace_{i\in I}$ is also a fusion Riesz  basis and $W_{i}\perp V_{j}$, for all $i\neq j$. Indeed
\begin{eqnarray*}
\langle Uf, (U^{-1})^{*}g\rangle =\langle f, g\rangle = 0,
\end{eqnarray*}
for every $f\in N_{i}$ and $g\in N_{j}$. Moreover, for every $\sigma \subset I$, $\{Uf_{i}\}\in \sum_{i\in I}\oplus W_{i}$ and $\{(U^{-1})^{*}g_{i}\}\in \sum_{i\in I}\oplus V_{i}$ we obtain
\begin{eqnarray*}
&&\left\Vert \sum_{i\in \sigma}Uf_{i}+\sum_{i\in \sigma^{c}}(U^{-1})^{*}g_{i}\right\Vert^{2}\\
&=&\left\Vert \sum_{i\in \sigma}Uf_{i}\right\Vert^{2}+\left\Vert \sum_{i\in \sigma^{c}}(U^{-1})^{*}g_{i}\right\Vert^{2}+2Re \left\langle \sum_{i\in \sigma}Uf_{i}, \sum_{i\in \sigma^{c}}(U^{-1})^{*}g_{i}\right\rangle\\
&=&\left\Vert \sum_{i\in \sigma}Uf_{i}\right\Vert^{2}+\left\Vert \sum_{i\in \sigma^{c}}(U^{-1})^{*}g_{i}\right\Vert^{2}\\
&\geq& \Vert U^{-1}\Vert^{-2} \sum_{i\in \sigma}\Vert f_{i}\Vert^{2}+ \Vert U\Vert^{2}\sum_{i\in \sigma}\Vert g_{i}\Vert^{2}\\
&\geq& \beta\left( \sum_{i\in \sigma}\Vert f_{i}\Vert^{2}+\sum_{i\in \sigma}\Vert g_{i}\Vert^{2}\right)
\end{eqnarray*}
where $\beta=\min\{\Vert U^{-1}\Vert^{-2}, \Vert U\Vert^{2}\}$. Hence, the family $\{W_{i}\}_{i\in \sigma}\cup\{V_{i}\}_{i\in \sigma^{c}}$ is a  fusion Riesz sequence. It is sufficient to show that $\{UN_{i}\}_{i\in \sigma}\cup\{(U^{-1})^{*}N_{i}\}_{i\in \sigma^{c}}$ is a complete family in $\mathcal{H}$, for every $\sigma \subseteq I$. For this, suppose that  there exists $\sigma \subseteq I$ so that $\left(\overline{\operatorname{span}}\{UN_{i}\}_{i\in \sigma}\cup\{(U^{-1})^{*}N_{i}\}_{i\in \sigma^{c}} \right)^{\perp}\neq 0$. Thus, there is $0\neq f\in \left(\overline{\operatorname{span}}\{UN_{i}\}_{i\in \sigma}\cup\{(U^{-1})^{*}N_{i}\}_{i\in \sigma^{c}} \right)^{\perp}$. On the other hand, there is a unique sequence $\{\varphi_{i}\}\in \oplus N_{i}$ so that $f=\sum_{i\in I}U\varphi_{i}$ and so
\begin{eqnarray*}
0&=& \left\langle f, \sum_{i\in \sigma}U\varphi_{i}+\sum_{i\in \sigma^{c}}(U^{-1})^{*}\varphi_{i}\right\rangle\\
&=&\left\langle  \sum_{i\in \sigma}U\varphi_{i}+\sum_{i\in \sigma^{c}}U\varphi_{i}, \sum_{i\in \sigma}U\varphi_{i}+\sum_{i\in \sigma^{c}}(U^{-1})^{*}\varphi_{i}\right\rangle\\
&=&\left\Vert \sum_{i\in \sigma}U\varphi_{i}\right\Vert^{2}+ \left\Vert \sum_{i\in \sigma^{c}}\varphi_{i}\right\Vert^{2}.
\end{eqnarray*}
Hence,
\begin{eqnarray*}
f=  \sum_{i\in I}U\varphi_{i}= \sum_{i\in \sigma^{c}}U\varphi_{i}=U\sum_{i\in \sigma^{c}}\varphi_{i} = 0.
\end{eqnarray*}
This completes the proof.
\end{proof}

The authors in \cite{arefi1} showed that fusion Riesz bases unlike discrete Riesz bases have infinite many dual fusion frames.
In the following   theorem we prove that fusion Riesz bases are woven with all of their duals.
\begin{thm}\label{alter dual}
Let $V=\{(V_{i},\omega_{i})\}_{i\in I}$ be every dual of a fusion Riesz basis $W=\{(W_{i},\omega_{i})\}_{i\in I}$ in  $\mathcal{H}$. Then $W$ and $V$ are woven.
\end{thm}
\begin{proof}
First, we show that $W$ is woven with its canonical dual. For this,
suppose  $\{f_{i}\}_{i\in I}$ and $\{S_{W}^{-1}g_{i}\}_{i\in I}$ are two sequences in $\sum_{i\in I}\oplus W_{i}$ and $\sum_{i\in I}\oplus S_{W}^{-1}W_{i}$, respectively. Also, let
$\sigma\subset I$ then
\begin{eqnarray*}
&&\left\Vert \sum_{i\in \sigma}\omega_{i}f_{i}+\sum_{i\in \sigma^{c}}\omega_{i}S_{W}^{-1}g_{i}\right\Vert^{2}\\
&=&\left\Vert \sum_{i\in \sigma}\omega_{i}f_{i}\right\Vert^{2}+\left\Vert \sum_{i\in \sigma^{c}}\omega_{i}S_{W}^{-1}g_{i}\right\Vert^{2}+2Re \left\langle \sum_{i\in \sigma}\omega_{i}f_{i}, \sum_{i\in \sigma^{c}}\omega_{i}S_{W}^{-1}g_{i}\right\rangle\\
&=&\left\Vert \sum_{i\in \sigma}\omega_{i}f_{i}\right\Vert^{2}+\left\Vert \sum_{i\in \sigma^{c}}\omega_{i}S_{W}^{-1}g_{i}\right\Vert^{2}\\
&\geq&\alpha_{\sigma} \sum_{i\in \sigma}\left\Vert f_{i}\right\Vert^{2}+ \alpha^{'}_{\sigma} \sum_{i\in \sigma^{c}}\left\Vert S_{W}^{-1}g_{i}\right\Vert^{2}\\
&\geq&\beta_{\sigma}\left(  \sum_{i\in \sigma}\left\Vert f_{i}\right\Vert^{2}+  \sum_{i\in \sigma^{c}}\left\Vert S_{W}^{-1}g_{i}\right\Vert^{2}\right),
\end{eqnarray*}
where $\alpha_{\sigma}$, $\alpha^{'}_{\sigma}$ are the respective lower Riesz bounds of $\{(W_{i},\omega_{i})\}_{i\in \sigma}$ and $\{(S_{W}^{-1}W_{i},\omega_{i})\}_{i\in \sigma^{c}}$, and $\beta_{\sigma}=min\left\{\alpha_{\sigma}, \alpha^{'}_{\sigma}\right\}$.
This implies that , $\{W_{i}\}_{i\in \sigma}\cup\{S_{W}^{-1} W_{i}\}_{i\in \sigma^{c}}$ is a fusion Riesz  sequence.  Now, suppose that $f\in \mathcal{H}$ so that $f\perp \{W_{i}\}_{i\in \sigma}\cup\{S_{W}^{-1} W_{i}\}_{i\in \sigma^{c}}$ then
\begin{eqnarray*}
\Vert f\Vert^{2} &=& \langle f, f\rangle = \left\langle f,  \sum_{i\in I}\omega_{i}S_{W}^{-1}\pi_{W_{i}}f \right\rangle \\
&=& \left\langle f,  \sum_{i\in \sigma}\omega_{i}S_{W}^{-1}\pi_{W_{i}}f \right\rangle+ \sum_{i\in \sigma^{c}}\left\langle \pi_{S_{W}^{-1}W_{i}} f,  \omega_{i}S_{W}^{-1}\pi_{W_{i}}f \right\rangle =0.
 \end{eqnarray*}
Hence, $\{W_{i}\}_{i\in
\sigma}\cup\{S_{W}^{-1}W_{i}\}_{i\in \sigma^{c}}$ is a complete sequence for all $\sigma\subseteq I$. Thus, $W$ and $S_{W}^{-1} W$ are weakly woven and so are woven with a universal lower bound $\mathcal{C}$ by Theorem \ref{weakly equ}.
Now, let  $V=\{(V_{i},\omega_{i})\}_{i\in I}$  be a dual fusion frame of $W$. Then
 for all $i\in I$ we have  $S_{W}^{-1} W_{i}\subset V_{i}$,  by Corollary 2.6 of \cite{arefi1}. So,
\begin{eqnarray*}
\mathcal{H}= \overline{\operatorname{span}}\{W_{i}\}_{i\in \sigma}\cup\{S_{W}^{-1}W_{i}\}_{i\in \sigma^{c}} \subseteq  \overline{\operatorname{span}}\{W_{i}\}_{i\in \sigma}\cup\{V_{i}\}_{i\in \sigma^{c}},
\end{eqnarray*}
for every $\sigma\subseteq I$. Moreover,
\begin{eqnarray*}
&&\sum_{i\in \sigma}\omega_{i}^{2}\Vert\pi_{W_{i}}f\Vert^{2}+\sum_{i\in \sigma^{c}}\omega_{i}^{2}\Vert\pi_{V_{i}}f \Vert^{2}\\
&\geq& \sum_{i\in \sigma}\omega_{i}^{2}\Vert \pi_{W_{i}}f\Vert^{2}+\sum_{i\in \sigma^{c}}\omega_{i}^{2}\Vert\pi_{S_{W}^{-1}W_{i}}f\Vert^{2}\geq \mathcal{C} \Vert f\Vert^{2}. \\
\end{eqnarray*}
Thus, $W$ and $V$ are woven, as desired.
\end{proof}
Applying the above theorem and with the aid of Theorem 2.5 of \cite{arefi1} we immediately obtain the next result.
\begin{cor}\label{psivw}
Let $W=\{(W_{i},\omega_{i})\}_{i\in I}$ be a  fusion Riesz basis  for $\mathcal{H}$ and $V=\{(V_{i},\omega_{i})\}_{i\in I}$ be an approximate dual fusion frame of $W$. Then $W$ and $\lbrace (\psi_{vw}^{-1} V_{i},
\omega_{i})\rbrace_{i\in I}$  are woven, in which $\psi_{vw} =  \sum_{i\in I}  \omega_{i}^{2}\pi_{V_{i}} S_{W}^{-1} \pi_{W_{i}}$.
\end{cor}
\begin{ex}
Consider
\begin{eqnarray*}
W_{1} = \mathbb{R}^{2} \times \{0\}, \quad W_{2} = {\textit{span}}\{(0,0,1)\}.
\end{eqnarray*}
Then $W = \lbrace W_{i}\rbrace_{i=1}^{2}$ is a fusion Riesz basis of $\mathbb{R}^{3}$. Also, let
\begin{eqnarray*}
V_{1} = \mathbb{R}^{2} \times \{0\}, \quad V_{2} = {\textit{span}}\{(0,1/2,1)\}.
\end{eqnarray*}
Since $\Vert I_{\mathbb{R}^{3}}-\psi_{vw} \Vert < 1$ so
 $V = \lbrace V_{i}\rbrace_{i=1}^{2}$ is an approximate dual fusion frame of $W$. A straightforward computation shows that
\begin{equation*}
\psi_{vw}^{-1}= \left[
 \begin{array}{ccc}

1 \qquad 0 \qquad 0\\

0 \qquad 1 \qquad 0\\

0 \quad -1/2 \quad 5/4\\

\end{array} \right].
\end{equation*}
Therefore,
\begin{eqnarray*}
\psi_{vw}^{-1} V_{1} = {\textit{span}}\{(1,0,0), (0,1,-1/2)\}, \quad \psi_{vw}^{-1} V_{2} = {\textit{span}}\{(0,1/2,1)\},
\end{eqnarray*}
and consequently $W$ is woven with $\lbrace \psi_{vw}^{-1} V_{i}\rbrace_{i=1}^{2}$, by Corollary \ref{psivw}.
\end{ex}
The next result,  gives a sufficient condition, under which a fusion frame and its dual are woven. The proof is similar to Theorem \ref{alter dual} so we regardless of the proof only  state the result.
\begin{cor}
Let $W=\{(W_{i},\omega_{i})\}_{i\in I}$ be a  fusion frame for $\mathcal{H}$ and $V=\{(V_{i},\omega_{i})\}_{i\in I}$ be a dual fusion frame of $W$ so that $\{W_{i}\}_{i\in \sigma}\cup\{V_{i}\}_{i\in \sigma^{c}}$ is a fusion frame  sequence for all $\sigma\subset I$. Then $W$ and $V$ are woven.
\end{cor}

\section{Weaving fusion frames and operator perturbations}
Linear perturbation of fusion frames, that is obtaining  some conditions on a fusion frame $W=\{(W_{i},\omega_{i})\}_{i\in I}$ and a linear operator $T\in B(\mathcal{H})$ so that  $\{(\overline{TW_{i}},\omega_{i})\}_{i\in I}$ constitutes a fusion frame for $\mathcal{H}$ or $T(\mathcal{H})$,  is one of striking problems in fusion frame setting. Some results in this issue can be found in \cite{As, Gav07, Li, Ruiz}.
In this section, we obtain some new results on  fusion frames and weaving fusion frames under operator perturbations. To this end, we need to the notion of  Friedrichs angle between two closed subspaces, reduced minimum modulus of   bounded linear operators and some basic results.
\begin{defn}
Given two  closed subspaces $M$ and $N$ of a Hilbert space $\mathcal{H}$, the angle between $M$ and $N$ is the angle in $[0, \pi/2]$ whose cosine is defined by
\begin{eqnarray*}
c(M,N)=\sup\{\vert \langle x, y \rangle\vert : \quad x\in M\ominus N, y\in N\ominus M, \Vert x\Vert=\Vert y\Vert =1\}.
\end{eqnarray*}
\end{defn}
Also, for an operator $T\in B(\mathcal{H})$ its reduced minimum modulus is defined by
\begin{eqnarray*}
\gamma(T)=\inf\{\Vert Tx\Vert: \quad \Vert x\Vert=1, x\in N(T)^{\perp}\}.
\end{eqnarray*}
It is well known that $\gamma(T)>0$ if and only if $T$ is a closed range operator and in this case
$\gamma(T)=\gamma(T^{*})=\gamma(T^{*}T)^{1/2}=\Vert T^{\dagger}\Vert^{-1}$, where $T^{\dagger}$ is the Moore-Penrose pseudoinverse \cite{Ding}. The next result determines an important connection between
angles and reduced minimum modulus of closed range operators.
\begin{prop}\label{modulus}\cite{Antezana}
Suppose $T\in B(\mathcal{H})$ is a closed range operator and $V$ is  a closed subspace of $\mathcal{H}$. Also, let $c:=c(N(T), V)<1$. Then
\begin{eqnarray*}
\gamma(T)(1-c^{2})^{1/2}\leq \gamma(T\pi_{V}) \leq \Vert T\Vert(1-c^{2})^{1/2}.
\end{eqnarray*}
\end{prop}
\begin{lem}\label{lem3}\cite{Gav07}
Let   $\mathcal{H}$ be a Hilbert spaces and $T\in B(\mathcal{H})$. Also, let $V$ be a closed subspace of $\mathcal{H}$. Then
\begin{eqnarray*}
\pi_{V}T^{*}=\pi_{V}T^{*}\pi_{\overline{TV}}.
\end{eqnarray*}
\end{lem}
Now, we are ready to state a necessary  and sufficient condition, under which image of  a  bounded operator on a given  family of closed subspaces  is a fusion frame for $\mathcal{H}$.
\begin{thm}\label{operator1}
Let $\{ W_{i}\}_{i\in I}$ be a family of closed subspaces in $\mathcal{H}$, $\{\omega_{i}\}_{i\in I}$ a family of weights
and $T\in B(\mathcal{H})$ be a closed range operator. Then the following conditions are equivalent;
\begin{itemize}
\item[(i)] The family $\{(T^{\dagger}TW_{i},\omega_{i})\}_{i\in I}$   is a  fusion frame  for $R(T^{*})$.
\item[(ii)] The family $\{(TW_{i},\omega_{i})\}_{i\in I}$   is a  fusion frame  for $\mathcal{H}$.
\end{itemize}
\end{thm}
\begin{proof}
First note that $T^{\dagger}TW_{i}$ is a closed subspace of $R(T^{*})$, for all $i\in I$. Moreover, for every $f\in N(T)$ and $g_{i}\in T^{\dagger}TW_{i}$ we obtain
\begin{eqnarray*}
\vert \langle f, g_{i}\rangle\vert = \vert \langle f, \pi_{R(T^{*})}g_{i}\rangle\vert = \vert \langle \pi_{R(T^{*})}f, g_{i}\rangle\vert =0,
\end{eqnarray*}
and so $c(N(T), T^{\dagger}TW_{i})=0$. Hence, by Proposition \ref{modulus}
\begin{eqnarray*}
\gamma(\pi_{T^{\dagger}TW_{i}}T^{*}) = \gamma(T\pi_{T^{\dagger}TW_{i}})\geq \gamma(T)>0,
\end{eqnarray*}
 for all $i\in I$, which implies that $ \gamma:=\inf_{i\in I}\gamma(\pi_{T^{\dagger}TW_{i}}T^{*})>0$. Now, suppose $f\in \mathcal{H}$ then by using Lemma \ref{lem3}
\begin{eqnarray*}
\sum_{i\in I}\omega_{i}^{2}\Vert \pi_{T^{\dagger}TW_{i}}T^{*}f\Vert^{2}&=& \sum_{i\in I}\omega_{i}^{2}\Vert \pi_{T^{\dagger}TW_{i}}T^{*}\pi_{TW_{i}}f\Vert^{2}\\
&\geq& \sum_{i\in I}\omega_{i}^{2}(\gamma(\pi_{T^{\dagger}TW_{i}}T^{*})^{2}\Vert \pi_{TW_{i}}f\Vert^{2}\\
&\geq&  \gamma^{2}\sum_{i\in I}\omega_{i}^{2}\Vert \pi_{TW_{i}}f\Vert^{2},
\end{eqnarray*}
where the first inequality comes from the fact that $\pi_{TW_{i}}f\in N(\pi_{T^{\dagger}TW_{i}}T^{*})^{\perp}$, for all $i\in I$. More precisely,  if $g\in N(\pi_{T^{\dagger}TW_{i}}T^{*})$ then $T^{*}g\in (\pi_{R(T^{*})}W_{i})^{\perp}$ and so for every $g_{i}\in W_{i}$
\begin{eqnarray*}
0=\langle T^{*}g,  \pi_{R(T^{*})}g_{i}\rangle=  \langle T^{*}g,  g_{i}\rangle= \langle g,  Tg_{i}\rangle.
\end{eqnarray*}
This implies that, for every $g\in N(\pi_{T^{\dagger}TW_{i}}T^{*})$  and $f\in \mathcal{H}$ we have that $\langle \pi_{TW_{i}}f,  g\rangle = 0$.  Moreover,  note that $\overline{TW_{i}}=TW_{i}$ is due to $\gamma(T\pi_{T^{\dagger}TW_{i}})>0$, for all $i\in I$.
On the other hand,
\begin{eqnarray*}
\sum_{i\in I}\omega_{i}^{2}\Vert \pi_{T^{\dagger}TW_{i}}T^{*}f\Vert^{2}&=& \sum_{i\in I}\omega_{i}^{2}\Vert \pi_{T^{\dagger}TW_{i}}T^{*}\pi_{TW_{i}}f\Vert^{2}\\
&\leq& \Vert T\Vert^{2}\sum_{i\in I}\omega_{i}^{2}\Vert \pi_{TW_{i}}f\Vert^{2}.
\end{eqnarray*}
Thus, $\{(T^{\dagger}TW_{i},\omega_{i})\}_{i\in I}$   is a  fusion frame  for $R(T^{*})$ if and only if $\{(TW_{i},\omega_{i})\}_{i\in I}$   is a  fusion frame  for $\mathcal{H}$.
\end{proof}
The following corollary is an immediate result of Theorem \ref{operator1}.
\begin{cor}
Suppose $\{ W_{i}\}_{i\in I}$ is  a family of closed subspaces in $\mathcal{H}$, $\{\omega_{i}\}_{i\in I}$ a family of weights and $T\in B(\mathcal{H})$ is a one to one and  closed range operator. Then $\{(W_{i},\omega_{i})\}_{i\in I}$   is a  fusion frame  for $\mathcal{H}$ if and only if $\{(TW_{i},\omega_{i})\}_{i\in I}$   is a  fusion frame  for $\mathcal{H}$.
\end{cor}
\begin{prop}
For an invertible operator $T$ on $\mathcal{H}$ and  a family of woven  fusion frames $\{(W_{ij},\omega_{ij})\}_{j=1,i\in I}^{M}$ with universal  bounds $C$ and $D$, the family $\{(TW_{ij},\omega_{ij})\}_{j=1,i\in I}^{M}$ is also woven with universal bounds $ \dfrac{C}{\Vert T^{-1}\Vert^{2}\Vert T\Vert^{2}}$ and $D\Vert T^{-1}\Vert^{2}\Vert T\Vert^{2}$.
 \end{prop}
\begin{proof}
Suppose   $T$ is an invertible operator and  $W$ a fusion frame with    bounds $C$ and $D$ then it is known that  $TW$ is also a fusion frame  for $\mathcal{H}$ with bounds  $ \dfrac{C}{\Vert T^{-1}\Vert^{2}\Vert T\Vert^{2}}$ and $D\Vert T^{-1}\Vert^{2}\Vert T\Vert^{2}$, see \cite{Cas08}. Now, since for every partition $\{\sigma_{j}\}_{j=1}^{M}$ of $I$ the sequence  $\{(W_{ij},\omega_{ij})\}_{j=1,i\in \sigma_{j}}^{M}$ is a fusion frame for $\mathcal{H}$ with universal  bounds $C$ and $D$, so the family  $\{(TW_{ij},\omega_{ij})\}_{j=1,i\in \sigma_{j}}^{M}$ is also  a fusion frame for $\mathcal{H}$ with the given bounds.
\end{proof}
The above proposition is an extension of  Proposition 11 of \cite{lynch},
however  unlike discrete frames
 the operator $T$ can not be changed by a  closed range or even  onto linear operator. Indeed, in fusion frames onto
operators  may not
preserve Besselian property, see \cite{Ruiz} for more details.
In the next theorem we present  some sufficient conditions under which a fusion frame and its perturbed by a bounded invertible operator constitute woven fusion frames.
\begin{thm}\label{per1} Let $W=\lbrace
(W_{i},\omega_{i})\rbrace_{i\in I}$  be a  fusion frames for $\mathcal{H}$ with respective bounds $C_{W}$ and $D_{W}$, also let $T\in B(\mathcal{H})$ be an invertible operator. Then the following hold;
 \begin{itemize}
\item[(i)] If either $W_{i}\subset TW_{i}$ or $TW_{i}\subset W_{i}$, then  $W$ and $TW=\lbrace
(TW_{i},\omega_{i})\rbrace_{i\in I}$ are woven.
\item[(ii)] If $W_{i}\subset T^{*}TW_{i}$ and  $\Vert I_{\mathcal{H}}-T^{-1}\Vert < \dfrac{C_{W}}{D_{W}}$,
then $W$ and $TW$ are woven.
\item[(iii)] If $T$ is a unitary operator so that $TS_{\sigma}-S_{\sigma}T$ is a positive operator, for every $\sigma\subset I$. Then $W$ and $TW$ are woven, in which $S_{\sigma}$ denotes the fusion frame operator of $W$ on the index set  $\sigma\subset I$.
\end{itemize}
 \end{thm}
\begin{proof}
First, we note that since $T$ is an invertible operator, so  $TW$ is a fusion frame  for $\mathcal{H}$.
To prove $(i)$   note that if  $W_{i}\subset TW_{i}$, then
\begin{eqnarray*}
\Vert f\Vert^{2}(D_{W}+D_{TW})&\geq&\sum_{i\in \sigma}\omega_{i}^{2}\Vert \pi_{W_{i}}f\Vert^{2}+\sum_{i\in \sigma^{c}}\omega_{i}^{2}\Vert\pi_{TW_{i}}f \Vert^{2}\\
&\geq& \sum_{i\in \sigma}\omega_{i}^{2}\Vert \pi_{W_{i}}f\Vert^{2}+\sum_{i\in \sigma^{c}}\omega_{i}^{2}\Vert \pi_{W_{i}}f\Vert^{2}\\
&\geq& C_{W}\Vert f\Vert^{2}.
\end{eqnarray*}
The case $TW_{i}\subset W_{i}$ can be  proved  similarly.
In order to show  $(ii)$ suppose $\{e_{i,j}\}_{j\in J_{i}}$ is an orthonormal basis of $W_{i}$, for all $i\in I$. Then $\{\omega_{i}e_{i,j}\}_{j\in J_{i}, i\in I}$ is a frame for $\mathcal{H}$ with respective bounds $C_{W}$ and $D_{W}$, \cite{Cas04}. Moreover, by the assumption  the sequences $\{(T^{*})^{-1}e_{i,j}\}_{j\in J_{i}}$ constitute local frames of $TW$ and so by   Proposition \ref{prop1.4} $\{\omega_{i}e_{i,j}\}_{j\in J_{i}, i\in I}$ and $\{\omega_{i}(T^{*})^{-1}e_{i,j}\}_{j\in J_{i}, i\in I}$  are woven. Thus, the result follows using Lemma \ref{1lem}.
 For proving $(iii)$, we note that since $T$ is a unitary operator so the fusion frame operator of $TW$ is $TS_{W}T^{*}$, see \cite{Gav07}. Hence, for every $\sigma \subset I$ and the weaving $\mathcal{V}=\lbrace
(W_{i},\omega_{i})\rbrace_{i\in \sigma}\cup \lbrace
(TW_{i},\omega_{i})\rbrace_{i\in \sigma^{c}}$ we obtain
\begin{eqnarray*}
S_{\mathcal{V}} &=& S_{\sigma}+TS_{\sigma^{c}}T^{*}\\
&=&S_{W}-S_{\sigma^{c}}+TS_{\sigma^{c}}T^{*}\\
&\geq& S_{W}.
\end{eqnarray*}
Thus $S_{\mathcal{V}} \geq S_{W}$ and this implies that $T^{*}_{\mathcal{V}}$ is an injective operator, i.e., $\mathcal{V}$ is a fusion frame,  which follows the result.
\end{proof}

\bibliographystyle{amsplain}

\end{document}